\documentclass[a4paper]{amsart}
\usepackage{amsthm,amsfonts,amsmath,amssymb}
\usepackage[abs]{overpic}
\usepackage{comment} 
\usepackage{color} 
\usepackage{bm}

\newtheorem{theorem}{Theorem}[section]
\newtheorem{lemma}[theorem]{Lemma}
\newtheorem{proposition}[theorem]{Proposition}
\newtheorem{corollary}[theorem]{Corollary}
\theoremstyle{definition}
\newtheorem{definition}[theorem]{Definition}
\newtheorem{remark}[theorem]{Remark}

\makeatletter

\@addtoreset{figure}{section}
\makeatother

\begin{document}

\title{Burnside groups and $n$-moves for links}
\author{Haruko A. Miyazawa, Kodai Wada and Akira Yasuhara}

\address{
Institute for Mathematics and Computer Science, Tsuda University,
2-1-1 Tsuda-Machi, Kodaira, Tokyo, 187-8577, Japan}
\email{aida@tsuda.ac.jp}

\address{Faculty of Education and Integrated Arts and Sciences, Waseda University, 1-6-1 Nishi-Waseda, Shinjuku-ku, Tokyo, 169-8050, Japan}
\email{k.wada8@kurenai.waseda.jp}

\address{
Department of Mathematics, Tsuda University, 
2-1-1 Tsuda-Machi, Kodaira, Tokyo, 187-8577, Japan}
\email{yasuhara@tsuda.ac.jp}



\thanks{This work was supported by JSPS KAKENHI Grant Numbers 
JP17J08186, JP17K05264.}


\begin{abstract}
Let $n$ be a positive integer. 
M. K. D\c{a}bkowski and J. H. Przytycki introduced the $n$th Burnside group of links which is preserved by $n$-moves, 
and proved that for any odd prime $p$ there exist links 
which are not equivalent to trivial links up to $p$-moves 
by using their $p$th Burnside groups. 
This gives counterexamples for the Montesinos-Nakanishi $3$-move conjecture. 
In general, it is hard to distinguish $p$th Burnside groups of a given link and a trivial link. 
We give a necessary condition for which $p$th Burnside groups are isomorphic to those of trivial links. 
The necessary condition gives us an efficient way to distinguish $p$th Burnside groups of a given link and a trivial link. 
As an application, 
we show that there exist links, each of which is not equivalent to a trivial link up to $p$-moves for any odd prime $p$. 
\end{abstract}

\maketitle

\section{Introduction}
Let $n$ be a positive integer. 
An {\em $n$-move} on a link is a local change as illustrated in
Figure~\ref{n-move}. 
Two links are {\em $n$-move equivalent} if they are transformed into each other by a finite sequence of $n$-moves. 
Note that if $n$ is odd then an $n$-move may change the number of components of a link. 
Since a $2$-move is generated by crossing changes and vice versa, 
we can consider an $n$-move as a generalization of a crossing change. 
Any link can be transformed into a trivial link by a finite sequence of crossing changes. 
Therefore, it is natural to ask whether or not any link is $n$-move equivalent to a trivial link. 
In $1980$s, Yasutaka~Nakanishi proved that 
all links with $10$ or less crossings and Montesinos links are $3$-move equivalent to trivial links, 
and he conjectured that
any link is $3$-move equivalent to a trivial link 
(see~\cite[Problem 1.59 (1)]{Kirby}). 
This conjecture is called the {\em Montesinos-Nakanishi $3$-move conjecture}, 
and have been shown to be true for several classes of links, 
for example, all links with $12$ or less crossings, closed $4$-braids and $3$-bridge links~\cite{C,P,PT}.

\begin{figure}[htbp]
  \begin{center}
    \begin{overpic}[width=11cm]{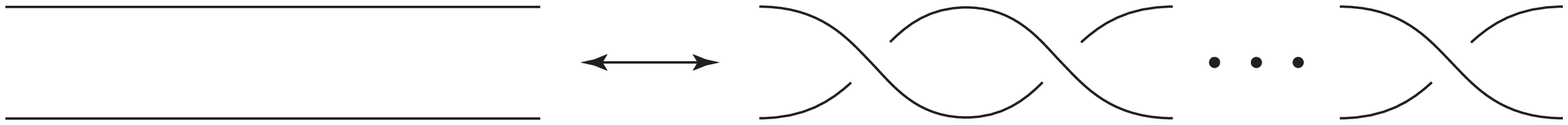}
      \put(172,-2){{\footnotesize $1$}}
      \put(210.5,-2){{\footnotesize $2$}}
      \put(288,-2){{\footnotesize $n$}}
      \put(113,16.5){$n$-move}
    \end{overpic}
  \end{center}
  \caption{}
  \label{n-move}
\end{figure}

After 20 years, in \cite{DP02,DP04} M. K. D\c{a}bkowski and J. H. Przytycki introduced the {\em $n$th Burnside group} of a link 
as an $n$-move equivalent invariant, 
and proved that for any odd prime $p$ there exist links 
which are not $p$-move equivalent to trivial links 
by using their $p$th Burnside groups. 
More precisely, they proved that the closure of the $5$-braid $(\sigma_{1}\sigma_{2}\sigma_{3}\sigma_{4})^{10}$ and the $2$-parallel of the Borromean rings are not $3$-move equivalent to trivial links \cite{DP02}, 
and that the closure of the $4$-braid $(\sigma_{1}\sigma_{2})^{6}$
is not $p$-move equivalent to a trivial link for any prime number $p\geq 5$ \cite{DP04}.
That is, they gave 
counterexamples for the Montesinos-Nakanishi $3$-move conjecture.

It is easy to see that the $p$th Burnside group is preserved by $p$-moves.  
While the $p$th Burnside group is a powerful invariant, 
it is hard to distinguish $p$th Burnside groups of given links in general. 
Hence to find a way to distinguish given Burnside groups is very important. 
In this paper, we give a necessary condition for which $p$th Burnside groups of links are 
isomorphic to those of trivial links (Lemma~\ref{th-FB}). 
The necessary condition gives us an efficient way to distinguish 
$p$th Burnside groups of a given link and a trivial link. 
As a consequence, we have 
a new obstruction to trivializing links by $p$-moves (Theorem~\ref{th-pmove}). 
In fact, by using Theorem~\ref{th-pmove}, 
we show that there exist links, each of which is
not $p$-move equivalent to a trivial link for any odd prime $p$ 
(Theorem~\ref{example}).
Our method is naturally extended to both {\em virtual} and {\em welded} links. 
We prove that there exists a welded link 
which is not $p$-move equivalent to a trivial link 
for any odd prime $p$ (Remark~\ref{ex-w4braid}).

\section{Free Burnside groups}
Throughout this paper, 
for a group $G$ let $\gamma_{q}G$ denote the $q$th term of the lower central series of $G$, 
that is, $\gamma_{1}G=G$ and $\gamma_{q+1}G=[\gamma_{q}G,G]$ $(q=1,2,\ldots)$. 

Let $F_{m}=\langle x_{1},\ldots,x_{m}\rangle$ be the free group of rank $m$.
We set $F(m,n)=F_{m}/\overline{W_{n}}$, 
where $\overline{W_{n}}$ is the normal subgroup of $F_{m}$ generated by $W_{n}=\{w^{n}\mid w\in F_{m}\}$ for a positive integer $n$. 
The group $F(m,n)$ is called the {\em $m$ generator free Burnside group of exponent $n$}.  
Let $F^{q}(m,n)$ denote the quotient group $F(m,n)/\gamma_{q}F(m,n)$. 
We remark that  
$F(m,n)$ is not always finite 
but $F^{q}(m,n)$ is a finite group for all $q$, 
see for example \cite[Chapter $2$]{V}.

\begin{lemma}
\label{th-FB}
Let $G$ be a group with a presentation $\langle x_{1},\ldots,x_{m}~\vline~R,W_{n}\rangle$, where $R$ is a set of words. 
If $G/\gamma_{q}G$ and $F^{q}(m,n)$ are isomorphic, 
then $R\subset \overline{W_{n}}\times\gamma_{q}F_{m}$ for any $q$.
\end{lemma}

\begin{proof}
First we note that 
\[F^{q}(m,n)\cong\langle x_{1},\ldots,x_{m}\mid W_{n}, \gamma_{q}F_{m}\rangle\]
and
\[G/\gamma_{q}G\cong\langle x_{1},\ldots,x_{m}\mid R,W_{n},\gamma_{q}F_{m}\rangle.\]
Consider a sequence of two natural projections $\psi$ and $\phi$:  
\[
F_{m}\overset{\psi}{\longrightarrow}
\langle x_{1},\ldots,x_{m}\mid W_{n}, \gamma_{q}F_{m}\rangle
\overset{\phi}{\longrightarrow}
\langle x_{1},\ldots,x_{m}\mid R, W_{n}, \gamma_{q}F_{m}\rangle.
\]
Then for any $r\in R$, $\phi(\psi(r))$ vanishes because $\psi(r)=r\overline{W_{n}}\times\gamma_{q}F_{m}$. 
So $\psi(r)\in \ker{\phi}$. 
On the other hand, we have that 
\[
| F^{q}(m,n)|=|\langle x_{1},\ldots,x_{m}\mid W_{n}, \gamma_{q}F_{m}\rangle|
=|G/\gamma_{q}G|\times|\ker{\phi} |.
\]
Since $F^{q}(m,n)$ is finite, $|\ker{\phi} |=1$. 
Therefore we have that $r\in \overline{W_{n}}\times\gamma_{q}F_{m}$. 
\end{proof}

\section{Burnside groups of links}
Let $L$ be a link in the $3$-sphere $S^{3}$ and $D$ an unoriented diagram of $L$. 
In~\cite{FR,J,Kelly,W}, a group $\Pi_{D}^{(2)}$ of $D$ is defined 
as follows:  
Each arc of $D$ yields a generator, and each crossing of $D$ gives a relation $yx^{-1}yz^{-1}$, where $x$ and $z$ correspond to the underpasses and $y$ corresponds to the overpass at the crossing, 
see Figure~\ref{corerelation}. 
The group $\Pi_{D}^{(2)}$ is an invariant of $L$.
We call it the {\em associated core group} of $L$ and denote it by~$\Pi_{L}^{(2)}$.

\begin{figure}[htbp]
  \begin{center}
    \begin{overpic}[width=1.5cm]{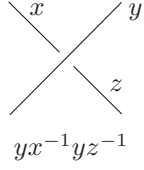}
      \put(8,38){$x$}
      \put(45,38){$y$}
      \put(38,9){$z$}
      \put(2,-15){$yx^{-1}yz^{-1}$}
    \end{overpic}
  \end{center}
  \vspace{1em}
  \caption{Relation of the associated core group}
  \label{corerelation}
\end{figure}

\begin{remark}
M. Wada~\cite{W} proved that $\Pi_{L}^{(2)}$ is isomorphic to the free product of the fundamental group of the double branched cover $M_{L}^{(2)}$ of $S^{3}$ branched along $L$ and the infinite cyclic group $\mathbb{Z}$.
That is, $\Pi_{L}^{(2)}\cong\pi_{1}(M_{L}^{(2)})*\mathbb{Z}$. 
Moreover, D\c{a}bkowski and Przytycki~\cite{DP02,DP04} pointed out that for a diagram $D$ of $L$,  
$\pi_{1}(M_{L}^{(2)})$ is obtained from 
the group $\Pi_{D}^{(2)}$ of $D$ by putting any fixed generator $x=1$. 
\end{remark}

In~\cite{DP02,DP04}, D\c{a}bkowski and Przytycki 
introduced $n$-move equivalence invariants of $L$ by using $\Pi_{L}^{(2)}$ and $\pi_{1}(M_{L}^{(2)})$ 
as follows.

\begin{definition}[\cite{DP02,DP04}] 
Suppose that $\Pi_{D}^{(2)}=\langle x_{1},\ldots,x_{m}\mid R\rangle$. 
Then $\pi_{1}(M_{L}^{(2)})\cong\langle x_{1},\ldots,x_{m}\mid R,x_{m}\rangle$. 
The {\em unreduced $n$th Burnside group $\widehat{B}_{L}(n)$} of $L$ is defined as $\langle x_{1},\ldots,x_{m}\mid R,W_{n}\rangle$. 
The {\em $n$th Burnside group ${B}_{L}(n)$} of $L$ is 
defined as $\langle x_{1},\ldots,x_{m}\mid R,x_{m},W_{n}\rangle$.
\end{definition}

\begin{proposition}[\cite{DP02,DP04}]
$\widehat{B}_{L}(n)$ and ${B}_{L}(n)$ are preserved by $n$-moves. 
\end{proposition}

We will focus on the unreduced $n$th Burnside group $\widehat{B}_{L}(n)$ from now on. 
Let $\widehat{B}^{q}_{L}(n)$ denote the quotient group $\widehat{B}_{L}(n)/\gamma_{q}\widehat{B}_{L}(n)$, 
which is a finite group for all $q$.  
Then the proposition above immediately implies the following corollary.

\begin{corollary}
\label{cor-inv}
$\widehat{B}^{q}_{L}(n)$ and $|\widehat{B}^{q}_{L}(n)|$ are preserved by $n$-moves for any $q$.
\end{corollary}

\begin{remark}
\label{rem-coloring}
Let $\mathbb{Z}_{n}$ denote the cyclic group $\mathbb{Z}/n\mathbb{Z}$ of order $n$.
Let $L$ be a link and $D$ a diagram of $L$. 
A map $f:\{\text{arcs of $D$}\}\rightarrow\mathbb{Z}_{n}$ is 
a {\em Fox $n$-coloring} of $D$ if $f$ satisfies $f(x)+f(z)=2f(y)$ for each crossing of $D$, 
where $x$ and $z$ correspond to the underpasses and $y$ corresponds to the overpass at the crossing. 
The set of Fox $n$-colorings of $D$ forms an abelian group and is an invariant of $L$. 
It is known that the order of the abelian group 
is equal
to $|\widehat{B}^{2}_{L}(n)|$ (see~\cite[Proposition~4.5]{P16}). 
Moreover, if $L$ is the $m$-component trivial link, then 
$\widehat{B}^{2}_{L}(n)\cong\mathbb{Z}_n^m$.
\end{remark}

\section{Obstruction to trivializing links by $p$-moves}

Let $p$ be a prime number. 
The {\em Magnus $\mathbb{Z}_{p}$-expansion $E^{p}$} is a homomorphism 
from $F_{m}$ into the formal power series ring in non-commutative variables $X_{1},\ldots,X_{m}$ with $\mathbb{Z}_{p}$ coefficients defined by 
$E^{p}(x_{i})=1+X_{i}$ and $E^{p}(x_{i}^{-1})=1-X_{i}+X_{i}^{2}-X_{i}^{3}+\cdots$ $(i=1,\ldots,m)$.  
Then we have the following theorem.

\begin{theorem}
\label{th-pmove}
Let $L$ be a link with $\Pi_{L}^{(2)}\cong\langle x_{1},\ldots,x_{m}\mid R\rangle$ and $\widehat{B}^{2}_{L}(p)\cong\mathbb{Z}^{m}_{p}$. 
If $L$ is $p$-move equivalent to a trivial link,  
then for any $r\in R$,
\[
E^{p}(r)=1+\displaystyle\sum_{(i_{1},\ldots,i_{p})}c(i_{1},\ldots,i_{p})X_{i_{1}}\cdots X_{i_{p}}+d(p+1)
\] 
for some $c(i_{1},\ldots,i_{p})\in\mathbb{Z}_{p}$ such that 
$c(i_{1},\ldots,i_{p})=c(i_{\sigma(1)},\ldots,i_{\sigma(p)})$ 
for any  permutation $\sigma$ of $\{1,\ldots,p\}$, 
where $(i_{1},\ldots,i_{p})$ runs over $\{1,\ldots,m\}^{p}$ 
and $d(k)$ denotes the terms of degree $\geq k$. 
\end{theorem}

\begin{proof}
If $L$ is $p$-move equivalent to a trivial link $T$, 
then $\widehat{B}^{2}_{T}(p)\cong\widehat{B}_{L}^{2}(p)\cong\mathbb{Z}^{m}_{p}$. 
By Remark~\ref{rem-coloring} the number of components of $T$ is $m$, and hence $\widehat{B}_{L}^{q}(p)\cong F^{q}(m,p)$ 
for any positive integer $q$. 
Thus we have that 
$R\subset \overline{W_{p}}\times\gamma_{q}F_{m}$ by Lemma~\ref{th-FB}.   
In~particular, for any $r\in R$, 
we have that $r=(\prod_{j}u_{j}^{p})v$ for some $\prod_{j}u_{j}^{p}\in \overline{W_{p}}$ and $v\in\gamma_{q}F_{m}$ 
because $\overline{W_{p}}$ is a verbal subgroup.

Now we may assume that $q\geq p+1$. 
Hence we have that $E^{p}(v)=1+d(p+1)$. 
For each $j$, $E^{p}(u_{j})$ can be written in the form 
$1+\sum_{i=1}^{m}a_{ji}X_{i}+d(2)$ 
for some $a_{ji}\in\mathbb{Z}_{p}$. 
Then we have that 
\[
E^{p}(u^{p}_{j})=1+\left(\displaystyle\sum_{i=1}^{m}a_{ji}X_{i}+d(2)\right)^{p}
=1+\displaystyle\sum_{(i_{1},\ldots,i_{p})}a_{ji_{1}}\cdots a_{ji_{p}}
X_{i_{1}}\cdots X_{i_{p}}+d(p+1).
\] 
Thus we have that 
\[\begin{array}{lll}
E^{p}\left(\displaystyle\prod_{j}u_{j}^{p}\right)
&=&\displaystyle\prod_{j}\left(1+\sum_{(i_{1},\ldots,i_{p})} 
a_{ji_{1}}\cdots a_{ji_{p}}X_{i_{1}}\cdots X_{i_{p}}+d(p+1)\right) 
\vspace{0.5em} \\ 
&=&1+\displaystyle\sum_{(i_{1},\ldots,i_{p})}c(i_{1},\ldots, i_{p})X_{i_{1}}\cdots X_{i_{p}}+d(p+1) 
\end{array}
\]
for some $c(i_{1},\ldots,i_{p})\in\mathbb{Z}_{p}$ such that 
$c(i_{1},\ldots,i_{p})=c(i_{\sigma(1)},\ldots,i_{\sigma(p)})$ 
for any permutation $\sigma$ of $\{1,\ldots,p\}$.  
Therefore $E^{p}(r)$ is the desired form.  
\end{proof}

Even though $4$ is not prime, 
we can show
the following theorem by a similar way to the proof of Theorem~\ref{th-pmove}.
We note that $4$-moves preserve the number of components of a link.

\begin{theorem}
\label{th-4move}
Let $L$ be an $m$-component link with $\Pi_{L}^{(2)}\cong\langle x_{1},\ldots,x_{m}\mid R\rangle$. 
If $L$ is $4$-move equivalent to a trivial link, 
then for any $r\in R$, 
\[
E^{2}(r)=1+\displaystyle\sum_{(i_{1},i_{2},i_{3},i_{4})}c(i_{1},i_{2},i_{3},i_{4})X_{i_{1}}X_{i_{2}}X_{i_{3}}X_{i_{4}}+d(5)
\] 
for some $c(i_{1},i_{2},i_{3},i_{4})\in\mathbb{Z}_{2}$ such that 
$c(i_{1},i_{2},i_{3},i_{4})=c(i_{\sigma(1)},i_{\sigma(2)},i_{\sigma(3)},i_{\sigma(4)})$ for any permutation $\sigma$ of $\{1,2,3,4\}$, 
where $(i_{1},i_{2},i_{3},i_{4})$ runs over $\{1,\ldots,m\}^{4}$.
\end{theorem}

By applying Theorem~\ref{th-pmove}, we have the following theorem.

\begin{theorem}\label{example}
The closure of the $5$-braid 
$(\sigma_{1}\sigma_{2}\sigma_{3}\sigma_{4})^{10}$ and the $2$-parallel of the Borromean rings 
are not $p$-move equivalent to trivial links for any odd prime $p$. 
\end{theorem}

\begin{remark}
D\c{a}bkowski and Przytycki proved Theorem~\ref{example} for $p=3$  \cite[Theorem~6]{DP02}. 
In their proof, the condition that $p=3$ is essential, and hence  
it seems hard to show Theorem~\ref{example} by using their arguments. 
\end{remark}

\begin{proof}[Proof of Theorem~\ref{example}]
Let $\gamma$ be the $5$-braid 
$(\sigma_{1}\sigma_{2}\sigma_{3}\sigma_{4})^{10}$ described by a diagram in Figure~\ref{5braid}. 
We put labels $x_{i}$ $(i=1,2,3,4,5)$ on initial arcs of the diagram. 
Progress from left to right, 
then the arcs are labeled by using relations of the associated core group.  
Thus we obtain labels $Q_{i}$ of terminal arcs of $\gamma$ as follows (see~\cite[Lemma 5]{DP02}):
\[
Q_{i}=x_{1}x_{2}^{-1}x_{3}x_{4}^{-1}x_{5}
x_{1}^{-1}x_{2}x_{3}^{-1}x_{4}x_{5}^{-1}x_{i}
x_{5}^{-1}x_{4}x_{3}^{-1}x_{2}x_{1}^{-1}
x_{5}x_{4}^{-1}x_{3}x_{2}^{-1}x_{1}. 
\] 
Let $\overline{\gamma}$ be the closure of $\gamma$. 
Since we have relations $Q_{i}x_{i}^{-1}$ for $\Pi^{(2)}_{\overline{\gamma}}$,
$\Pi_{\overline{\gamma}}^{(2)}$ has the presentation
$\langle x_{1},x_{2},x_{3},x_{4},x_{5}
\mid r_{1},r_{2},r_{3},r_{4},r_{5}\rangle$, 
where $r_{i}=Q_{i}x_{i}^{-1}$. 
We note that $\widehat{B}^{2}_{\overline{\gamma}}(p)\cong\mathbb{Z}_{p}^{5}$ for any odd prime $p$. 
On the other hand, by computing $E^{p}(r_{1})$, then 
the coefficient of $X_{2}X_{3}X_{4}$ is $0$
and that of $X_{4}X_{2}X_{3}$ is $2$ in $E^{p}(r_{1})$. 
Theorem~\ref{th-pmove} implies that $\overline{\gamma}$ is not $p$-move equivalent to 
a trivial link.

\begin{figure}[htbp]
  \begin{center}
    \begin{overpic}[width=11cm]{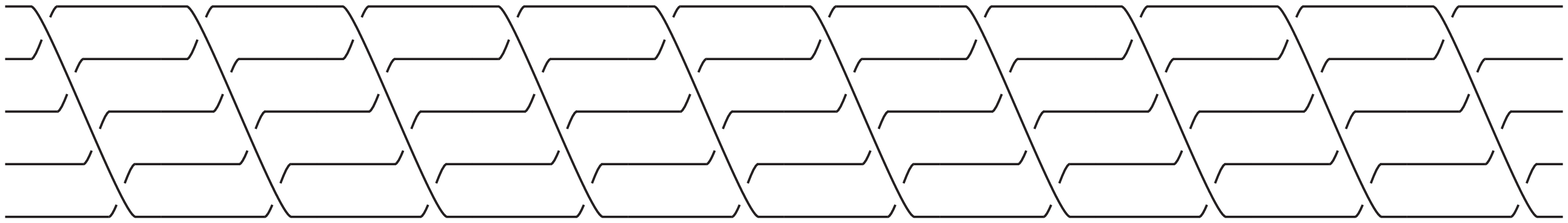}
      \put(-14,40){$x_{1}$}
      \put(-14,30){$x_{2}$}
      \put(-14,20){$x_{3}$}
      \put(-14,9){$x_{4}$}
      \put(-14,-2){$x_{5}$}
      \put(319,40){$Q_{1}$}
      \put(319,29){$Q_{2}$}
      \put(319,18){$Q_{3}$}
      \put(319,6.75){$Q_{4}$}
      \put(319,-4){$Q_{5}$}
    \end{overpic}
  \end{center}
  \caption{$5$-braid $\gamma=(\sigma_{1}\sigma_{2}\sigma_{3}\sigma_{4})^{10}$}
  \label{5braid}
\end{figure}


Let $\gamma'$ be the $6$-braid described by a diagram in Figure~\ref{6braid}. 
We put labels $x_{i}$ on initial arcs, $y_{i}$ on terminal arcs,  
and $Q_{i}$ on arcs of the diagram as illustrated in Figure~\ref{6braid} $(i=1,2,3,4,5,6)$.
By using relations of the associated core group, 
the labels $Q_{i}$ are expressed as follows: 
\[
Q_{i}=
\left\{\begin{array}{ll}
x_1x_2^{-1}x_5x_6^{-1}x_2x_1^{-1}x_ix_1^{-1}x_2x_6^{-1}x_5x_2^{-1}x_1 &\\
\hspace*{.5em}=y_1y_2^{-1}y_3y_4^{-1}y_5y_6^{-1}y_4y_3^{-1}y_2y_1^{-1}y_i
y_1^{-1}y_2y_3^{-1}y_4y_6^{-1}y_5y_4^{-1}y_3y_2^{-1}y_1&\hspace*{-.7em}(i=1,2),\\
x_6x_5^{-1}x_ix_5^{-1}x_6&\\
\hspace*{.5em}=Q_6Q_5^{-1}y_iQ_5^{-1}Q_6
=x_1x_2^{-1}x_6x_5^{-1}x_2x_1^{-1}y_ix_1^{-1}x_2x_5^{-1}x_6x_2^{-1}x_1&\hspace*{-.7em}(i=3,4),\\
x_1x_2^{-1}x_ix_2^{-1}x_1=y_4y_3^{-1}y_1y_2^{-1}y_3y_4^{-1}y_iy_4^{-1}y_3y_2^{-1}y_1y_3^{-1}y_4&\hspace*{-.7em}(i=5,6).
\end{array}\right.
\] 
Since the closure of $\gamma'$ is the $2$-parallel of the Borromean rings $L_{2BR}$,  
 $\Pi_{L_{2BR}}^{(2)}$ has the presentation
$\langle x_{1},x_{2},x_{3},x_{4},x_{5},x_{6}
\mid r_{1},r_{2},r_{3},r_{4},r_{5},r_{6}\rangle$, where 
\[
r_i=\left\{
\begin{array}{ll}
(x_1x_2^{-1}x_5x_6^{-1}x_2x_1^{-1}x_ix_1^{-1}x_2x_6^{-1}x_5x_2^{-1}x_1)^{-1}&\\
~\times x_1x_2^{-1}x_3x_4^{-1}x_5x_6^{-1}x_4x_3^{-1}x_2x_1^{-1}x_i
x_1^{-1}x_2x_3^{-1}x_4x_6^{-1}x_5x_4^{-1}x_3x_2^{-1}x_1&\hspace*{-.7em}(i=1,2),\\
(x_6x_5^{-1}x_ix_5^{-1}x_6)^{-1}
x_1x_2^{-1}x_6x_5^{-1}x_2x_1^{-1}x_i
x_1^{-1}x_2x_5^{-1}x_6x_2^{-1}x_1&\hspace*{-.7em}(i=3,4),\\
(x_1x_2^{-1}x_ix_2^{-1}x_1)^{-1}x_4x_3^{-1}x_1x_2^{-1}x_3x_4^{-1}x_ix_4^{-1}x_3x_2^{-1}x_1x_3^{-1}x_4&\hspace*{-.7em}(i=5,6).\end{array}
\right.
\] 
We note that $\widehat{B}^{2}_{L_{2BR}}(p)\cong\mathbb{Z}_{p}^{6}$ for any odd prime $p$.
On the other hand, by computing $E^{p}(r_{6})$, then 
the coefficient of $X_{2}X_{4}X_{6}$ is $1$
and that of {$X_{4}X_{6}X_{2}$} is $0$ in $E^{p}(r_{6})$. 
Theorem~\ref{th-pmove} implies that $L_{2BR}$ is not $p$-move equivalent to 
a trivial link.  
\end{proof}

\begin{figure}[t]
  \begin{center}
    \begin{overpic}[width=6cm]{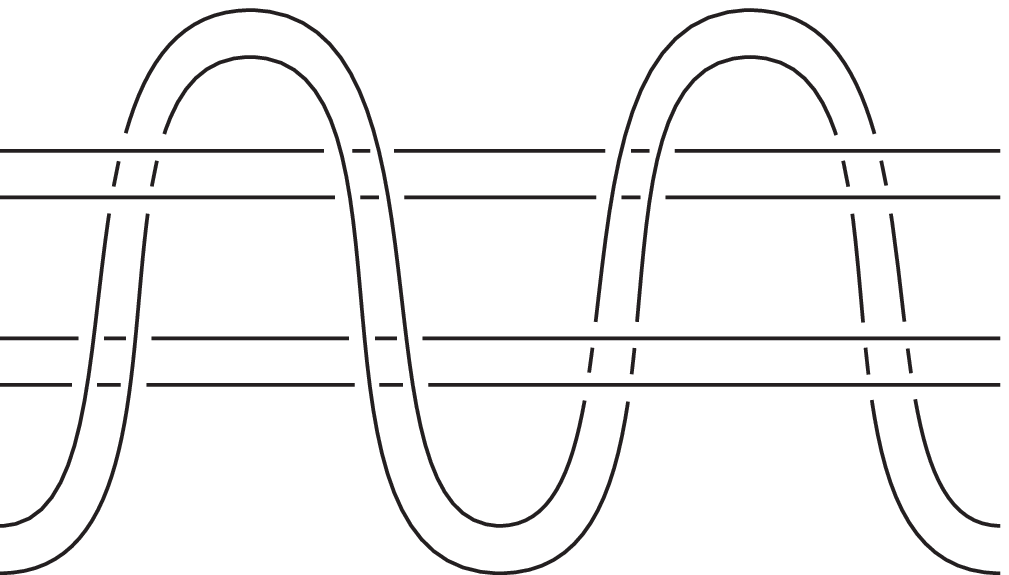}
      \put(-14,71){$x_{1}$}
      \put(-14,61){$x_{2}$}
      \put(-14,39){$x_{3}$}
      \put(-14,30){$x_{4}$}
      \put(-14,7){$x_{5}$}
      \put(-14,-2){$x_{6}$}
      \put(175,71){$y_{1}$}
      \put(175,61){$y_{2}$}
      \put(175,39){$y_{3}$}
      \put(175,30){$y_{4}$}
      \put(175,7){$y_{5}$}
      \put(175,-2){$y_{6}$}
      \put(80,76){$Q_{1}$}
      \put(80,55){$Q_{2}$}
      \put(37,44){$Q_{3}$}
      \put(37,23){$Q_{4}$}
      \put(80,14){$Q_{5}$}
      \put(80,-8.5){$Q_{6}$}
    \end{overpic}
  \end{center}
  \vspace{0.5em}
  \caption{$6$-braid $\gamma'$ whose closure is the $2$-parallel of the Borromean rings $L_{2BR}$}
  \label{6braid}
\end{figure}

\begin{remark}
\label{ex-w4braid}
For a welded link $L$, we can similarly define 
the {\em associated core group $\Pi_{L}^{(2)}$} and 
the {\em unreduced $n$th Burnside group $\widehat{B}_{L}(n)$} of $L$. 
We note that Theorems~\ref{th-pmove} and~\ref{th-4move} hold for welded links. 
Hence, we can show that 
there exists a welded link which is not $p$-move equivalent to a trivial link for any odd prime $p$ as follows.
Let $b$ be the welded $4$-braid described by a virtual diagram in Figure~\ref{w4braid}. 
We put labels $x_{i}$ and $Q_{i}$ $(i=1,2,3,4)$ on initial and terminal arcs of the diagram, respectively. 
By using relations of the associated core group, 
the labels $Q_{i}$ are expressed as follows: 
\[
Q_{i}=\left\{
\begin{array}{ll}
x_{4}x_{1}^{-1}x_{2}x_{4}^{-1}x_{1}x_{2}^{-1}x_{3}
x_{2}^{-1}x_{1}x_{4}^{-1}x_{2}x_{1}^{-1}x_{4} &\text{if $i=3$}, \\
x_{i} &\text{otherwise}.
\end{array}\right.
\]
Let $\overline{b}$ be the closure of $b$, 
then $\Pi_{\overline{b}}^{(2)}\cong\langle x_{1},x_{2},x_{3},x_{4}
\mid Q_{3}x_{3}^{-1}\rangle$. 
We note that $\widehat{B}^{2}_{\overline{b}}(p)\cong\mathbb{Z}_{p}^{4}$ for any odd prime $p$.  
On the other hand, by computing $E^{p}(Q_{3}x_{3}^{-1})$, 
we have that the coefficient of $X_{4}X_{2}X_{3}$ is $1$ and 
that of $X_{4}X_{3}X_{2}$ is $0$ in $E^{p}(Q_{3}x_{3}^{-1})$. 
Therefore, we have that 
$\overline{b}$ is not $p$-move equivalent to
a trivial link by Theorem~\ref{th-pmove}.
\end{remark}

\begin{remark}
All of the three links $\overline{\gamma}, L_{2BR}$ and $\overline{b}$ above are not $4$-move equivalent to trivial links 
by Theorem~\ref{th-4move}. 
Because terms of degree $3$ survive in $E^{2}(r)$ for some relation $r$ of $\Pi_{L}^{(2)}$ 
$(L=\overline{\gamma},L_{2BR},\overline{b})$. 
\end{remark}

\begin{figure}[t]
  \begin{center}
    \begin{overpic}[width=6cm]{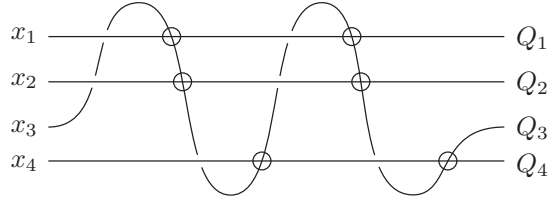}
      \put(-14,59){$x_{1}$}
      \put(-14,42){$x_{2}$}
      \put(-14,24){$x_{3}$}
      \put(-14,11){$x_{4}$}
      \put(175,57){$Q_{1}$}
      \put(175,40){$Q_{2}$}
      \put(175,23){$Q_{3}$}
      \put(175,9){$Q_{4}$}
    \end{overpic}
  \end{center}
  \caption{Welded $4$-braid $b$}
  \label{w4braid}
\end{figure}




\begin{thebibliography}{99}
\bibitem{C} Q. Chen, {\it The 3-move conjecture for 5-braids}, Knots in Hellas '98 (Delphi), 36--47, Ser. Knots Everything {\bf 24}, World Sci. Publ., River Edge, NJ (2000).

\bibitem{DP02} M. K. D\c{a}bkowski and J. H. Przytycki, {\it Burnside obstructions to the Montesinos-Nakanishi $3$-move conjecture}, Geom. Topol. {\bf 6} (2002), 355--360.

\bibitem{DP04} M. K. Dabkowski and J. H. Przytycki, {\it Unexpected connections between Burnside groups and knot theory}, Proc. Natl. Acad. Sci. USA {\bf 101} (2004), 17357--17360. 

\bibitem{FR} R. Fenn and C. Rourke, {\it Racks and links in codimension two}, J. Knot Theory Ramifications~{\bf 1} (1992), 343--406.

\bibitem{GPV} M. Goussarov; M. Polyak and O. Viro, {\it Finite-type invariants of classical and virtual knots}, Topology {\bf 39} (2000), 1045--1068.

\bibitem{J} D. Joyce, {\it A classifying invariant of knots, the knot quandle}, J. Pure Appl. Algebra {\bf 23} (1982), 37--65.

\bibitem{Kelly} A. J. Kelly, {\it Groups from link diagrams}, Ph.D. Thesis, U. Warwick (1990).

\bibitem{Kirby} R. Kirby, {\it Problems in low-dimensional topology; Geometric Topology}, from ''Proceedings of the Georgia International Topology Conference, 1993'', Studies in Advanced Mathematics {\bf 2} (W Kazez, Editor), AMS/IP (1997), 35--473.

\bibitem{P} J. H. Przytycki, {\it Elementary conjectures in classical knot theory}, Quantum topology, 292--320, Ser. Knots Everything {\bf 3}, World Sci. Publ., River Edge, NJ (1993). 

\bibitem{P16} J. H. Przytycki, {\it On Slavik Jablan's work on 4-moves}. J. Knot Theory Ramifications {\bf 25} (2016), 1641014, 26 pp.

\bibitem{PT} J. H. Przytycki and  T. Tsukamoto, {\it The fourth skein module and the Montesinos-Nakanishi conjecture for 3-algebraic links}, J. Knot Theory Ramifications {\bf 10} (2001), 959--982.

\bibitem{V} M. Vaughan-Lee, {\it The restricted Burnside problem}, second edition, London Mathematical Society Monographs, New Series {\bf 8}, The Clarendon Press, Oxford University Press, New York (1993).

\bibitem{W} M. Wada, {\it Group invariants of links}, Topology {\bf 31} (1992), 399--406.

\end{thebibliography}
\end{document}